\documentclass[11pt]{amsart}

\newtheorem{theorem}{Theorem}
\newtheorem{lemma}{Lemma}
\newtheorem{proposition}{Proposition}

\newtheorem{remark}{Remark}

\usepackage{graphicx,hyperref}
\usepackage{tikz}
\usetikzlibrary{arrows}

\newcommand{\Z}{\mathbb Z}
\newcommand{\Q}{\mathbb Q}
\newcommand{\inft}{
\begin{tikzpicture}[scale=0.03528, baseline={([yshift=-.5ex]current bounding box.center)}]
\draw (0,-13) arc [radius=9, start angle=-45, end angle=45];
\draw (10,0) arc [radius=9, start angle=135, end angle=225];
\end{tikzpicture}}
\newcommand{\zerot}{
\begin{tikzpicture}[scale=0.03528,baseline={([yshift=-.5ex]current bounding box.center)}]
\draw (0,-4) arc [radius=9, start angle=135, end angle=45];
\draw (12,6) arc [radius=9, start angle=-45, end angle=-135];
\end{tikzpicture}}

\begin{document}

\markboth{Tuzun and Sikora}
{Jones Unknot Conjecture}



\title{Verification Of The Jones Unknot Conjecture Up To 22 Crossings}
\author{Robert E. Tuzun, Adam S. Sikora}


\address{Department of Mathematics, University at Buffalo, Buffalo, NY  14260 \\
retuzun@buffalo.edu and asikora@buffalo.edu}

\maketitle

\begin{abstract}
We proved by computer enumeration that the Jones polynomial distinguishes the unknot for knots up to 22 crossings.  Following an approach of Yamada, we generated knot diagrams by inserting algebraic tangles into Conway polyhedra, computed their Jones polynomials by a divide-and-conquer method, and tested those with trivial Jones polynomials for unknottedness with the computer program SnapPy. We employed numerous novel strategies for reducing the computation time per knot diagram and the number of knot diagrams to be considered. That made computations up to 21 crossings possible on a single processor desktop computer. We explain these strategies in this paper. We also provide total numbers of algebraic tangles up to 18 crossings and of Conway polyhedra up to 22 vertices.
We encountered new unknot diagrams with no crossing-reducing pass moves in our search.  We report one such diagram in this paper.
\end{abstract}

\keywords{Jones unknot conjecture, Jones polynomial, Kauffman bracket, algebraic tangle, Conway polyhedron}

\subjclass[2010]{Mathematics Subject Classification 2000: 57M25, 57M27}

\section{Introduction}

One of the most prominent conjectures in knot theory is that of V. Jones asserting that the Jones polynomial distinguishes the unknot. Jones proposed it as one of the challenges for mathematics in the 21st century in \cite{Jo}.  Only limited cases of that conjecture have been verified so far:
\begin{itemize}
\item By \cite{Ka, Mu, Th1}, the span of the Jones polynomial of an alternating knot is its minimal crossing number. Consequently, the conjecture holds for alternating knots. More generally, it holds for all adequate knots, see e.g. \cite{AK}.
\item Hoste, Thistlethwaite, and Weeks \cite{HTW} tabulated all prime knots up to 16 crossings in the late 90s, showing no counter-examples to Jones' conjecture among them.
\item Dasbach and Hougardy \cite{DH} verified the conjecture through a computer search up to 17 crossings.
\item Yamada \cite{Ya} performed a computer-based verification of the Jones conjecture for knots up to 18 crossings and for algebraic knots up to 21 crossings. M. Thistlethwaite informs us that he tested the conjecture for all knots up to 21 crossings, \cite{Th3}.
\item Khovanov homology is a bi-graded homology theory for knots which refines (categorifies) the Jones polynomial. Kronheimer and Mrowka proved that these homology groups detect the unknot, \cite{KM}.
\end{itemize}

It is also worth pointing out that Bigelow \cite{Bi} related the Jones conjecture for knots of braid index $4$ to the question of the faithfulness of the Burau representation on the braid group $B_4$ on $4$ strands. Ito proved that non-faithfulness of the Burau representation on $B_4$ would indeed disprove the Jones conjecture, \cite{It}.
Additionally,  \cite{APR, JR, KR, Pr2, Ro} made attempts to disprove the Jones conjecture by considering generalized mutations on unknot diagrams, which preserve the Jones polynomial but potentially change their knot type.

Interestingly, the Jones conjecture does not generalize to links, as Thistlethwaite
\cite{Th2} found examples of 15-crossing two-component links with the
Jones polynomial of the unlink of 2 components. Later, Eliahou, Kauffman, and Thistlethwaite \cite{EKT} found infinite families of such links.

We prove

\begin{theorem}
The Jones polynomial distinguishes all non-trivial
knots up to 22 crossings from the unknot.
\end{theorem}


Obtaining this result required testing 2,257,956,747,340 non-algebraic knot
diagrams and 16,043,635,711 algebraic knot diagrams, for a total of
2,274,000,383,051 knot diagrams.

We achieved the above result through a three step approach, similar to that of Yamada \cite{Ya}:
\begin{enumerate}
\item Generation of appropriate knot diagrams up to 22 crossings, by (a) inserting
algebraic tangles into Conway polyhedra and by (b) considering closures of
algebraic tangles. Not all knot diagrams are necessary for the purpose of testing
of the Jones conjecture. We discuss these details in Sec. \ref{s_enumeration}.
\item Computation of the determinants of all diagrams and computation of the Kauffman bracket polynomials of the diagrams with determinant $1$, using a divide-and-conquer method,  see Sec. \ref{s_KB}.
\item Testing the diagrams with Kauffman brackets of the form $\pm A^{r}$ for possible crossing reducing pass moves first and then for unknottedness using the computer program SnapPy \cite{Sn}, see Sec. \ref{s_pass}-\ref{s_reco}.
\end{enumerate}

In all of these steps, we employed several novel enhancements for reducing the computation time and memory requirements resulting in reasonable computation times on a single processor desktop machine for computations up to 21 crossings, and on a computer cluster for 22 crossing computations.

Of the $2.27\cdot 10^{12}$ knot diagrams generated only 0.14\% had determinant $1$ and only a fraction $7.7 \cdot 10^{-7}$ had trivial Jones polynomial.

We observed that the computational effort for different parts of the Jones conjecture verification rises by a factor of between 5 and 10 in CPU time for every increment in number of crossings considered. Further details of the computational aspect of the project including a breakdown of the timing are presented in Sec. \ref{s_computation}.

We plan to verify the conjecture for 23+ crossing diagrams by further optimization and by parallelization of our algorithms.

\section*{Acknowledgments}

We would like to thank Heinz Kredel for his extensive help with the JAS algebra software used in this work, Nathan Dunfield for his help with SnapPy, and  Andrey Zabolotskiy for bringing \cite{BrMcK} to our attention. We would also like to thank the Center for Computational Research at the University at Buffalo for providing access to their computer cluster on which the parallel computations were performed.



\section{Algebraic Tangles and Their Kauffman brackets}
\label{s_enumeration}
An {\em $n$-tangle} in a $3$-ball $B^3$ with distinguished $2n$ points $b_1,...,b_{2n}$ is a proper embedding of a $1$-manifold $L$ into $B^3$ such that $L\cap \partial B^3=\{b_1,...,b_{2n}\}$.
Tangles are considered up to isotopy in $B^3$ fixing $\partial B^3.$
We will refer to $2$-tangles simply as {\em tangles} and denote their endpoints by NW, NE, SE, and SW, following \cite{Co}.



We will use the standard coordinate system of $\mathbb R^3$ with $X,Y,Z$ axes pointing to the right, up, and towards the reader, respectively.
The $180^{o}$ rotations $s_X,s_Y,s_Z$ of tangles with respect to $X,$ $Y$, and $Z$ axes are called {\em tangle mutations}.

We will use the operations of tangle addition, multiplication, and reflection.  {\em The reflection of $T$,}
$r(T)$, is obtained by reflection of $T$
about the NW-SE axis, sending $(x,y,z)$ to $(-y,-x,z)$.  The {\em sum of tangles,} $T_{1}+T_{2},$ is
obtained by joining the NE and SE corners of $T_{1}$ with
the NW and SW corners of $T_{2}$, respectively.

\begin{figure}
\center{\includegraphics[width=0.6in]{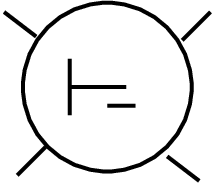}\quad
\includegraphics[width=1.2in]{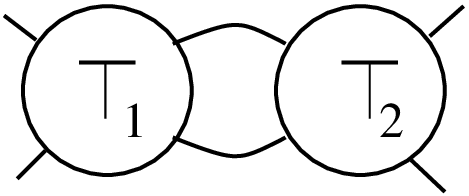}\quad
\includegraphics[width=0.9in]{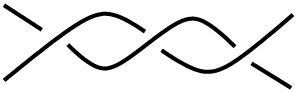}}
\caption{$ r(T_1), T_1+ T_2$, integral tangle $3$}
\label{tangle-fig}
\end{figure}


An {\em $n\in \Z$ (integral) tangle} is composed of $|n|$ half-twists, which are positive if $n>0$ and negative, otherwise, see Fig. \ref{tangle-fig}.
 (A twist is positive if its overstrand has a positive slope.)

Any tangle obtained from the zero tangle $\zerot$ by a sequence of reflections and additions of integer tangles is a {\em rational tangle,}  \cite{Co}.

Conway proved that if one associates the multiplicative inverse ($r\to 1/r$) with tangle reflections, then the value in $\hat \Q=\Q\cup \{\infty\}$ obtained in the process of building a rational tangle depends on the isotopy type of the tangle only, see \cite{Ad, Co, KL1}. Note that the reflection of \cite{KL1} differs from that of Conway by a mutation. However, \cite{KL1} prove that mutation preserves rational tangles.

Tangles formed by reflection and addition of rational tangles are
called {\em algebraic}. A closure of an algebraic tangle is an {\em algebraic} or {\em arborescent link}.

The following summarizes some known properties of algebraic tangles:

\begin{lemma}
If  $T$ is algebraic then so are:\\
(a) its mirror image, $\overline T$ obtained by reversing all crossings $T$\\
(b) mutants $s_X(T), s_Y(T),s_Z(T)$ of $T$.\\
(c) clockwise and counterclockwise $90^o$ rotation of $T$.
\end{lemma}

\begin{proof}
(a) $\overline T$ can be built from integer tangles the same way as $T$ by reversing signs of all integer tangles involved.\\
(b) Note that
$$s_X(T_1 + T_2)=s_X(T_1)+ s_X(T_2)\quad \text{and}\quad s_X(r(T))=s_Y(T)$$
$$s_Z(T_1 + T_2)=s_Z(T_2)+ s_Z(T_1)\quad \text{and}\quad s_Z(r(T))=r(s_Z(T_2)$$ and, finally,
$$s_Y=s_Xs_Z.$$
Now the statements follow by induction on the complexity of the tangle.\\
(c)
$90^o$ rotation results from the reflection $r$ followed by $s_X$, followed by the mirror image. All of these operations transform algebraic tangles to algebraic tangles.
\end{proof}

\section{Kauffman brackets}
Recall that the Kauffman bracket of framed links is defined recursively
by the following rules:
\begin{equation}
\langle O \rangle = 1
\end{equation}
\begin{equation}
\begin{tikzpicture}[scale=0.03528,baseline={([yshift=-.5ex]current bounding box.center)}]
\node at (-6,6) {$\langle$};
\draw (0,0) -- (12,12);
\draw (12,0) -- (8,4);
\draw (4,8) -- (0,12);
\node at (18,6) {$\rangle$};
\end{tikzpicture}
=A \langle \, \inft\, \rangle+A^{-1} \langle \, \zerot\, \rangle,
\end{equation}
\begin{equation}\label{trivial-loop}
\langle L \cup O \rangle = (-A^{2} - A^{-2}) \langle L \rangle
\end{equation}
where $L \cup O$ is any link with a trivial component $O.$

The Jones polynomial of a knot $K$ is a Laurent polynomial in one
variable, $t$, with integer coefficients, given by
\begin{equation}
J(D) = (-A^{-3})^{\omega(D)} \langle D \rangle |_{A=t^{-1/4}}
\end{equation}
where $\omega(D)\in \Z$ is the writhe of the knot diagram $D$.


Recall that \hspace*{-.05in}\zerot, and \inft are the zero and infinity tangles, respectively.
By resolving all crossings of a tangle $T$ by the above skein relations and eliminating all trivial components by (\ref{trivial-loop}), we arrive at
$$\langle T\rangle=a(A)\ \inft+b(A)\ \zerot$$
and call the pair of Laurent polynomials $(a(A), b(A))$ {\em the Kauffman bracket of $T$}.

Note that tangle mutations $s_X,s_Y,s_Z$ preserve Kauffman bracket while reflection $r(T)$ and the mirror image $\overline{T}$ transform $\langle T \rangle = (a,b)$ to
\begin{equation}\label{e-rot}
\langle r(T) \rangle = (\overline{b}, \overline{a}),\quad
\langle \overline{T} \rangle = (\overline{a}, \overline{b})
\end{equation}
where $\overline{a}$ and $\overline{b}$ denote $a(A^{-1})$ and $b(A^{-1})$
respectively.

The sum of tangles $T_{i}$ with
Kauffman brackets $(a_{i},b_{i})$, for $i = 1, 2$, has Kauffman bracket
\begin{equation}
\langle T_{1}+T_{2} \rangle =
(a_{1}b_{2} + b_{1}a_{2} - (A^{2}+A^{-2})a_{1}a_{2}, b_{1}b_{2})
\end{equation}

Let $t = A^{-4}$. The following lemma simplifies the storage of Kauffman brackets.

\begin{lemma}\label{KB-tangle}
For any tangle $T$, $\langle T\rangle=A^{n} (t^{-1/2} p, q)$ for some $n \in \Z$ and
some $p, q \in \mathbb{Z}[t^{\pm 1}]$.
\end{lemma}

\begin{proof}
If the diagram of $T$ has $n$ crossings then $\langle T\rangle$ is a sum of $2^n$ states, each of the form $A^c (-A^2-A^{-2})^l$\ \inft\hspace*{.1in}
or $A^c (-A^2-A^{-2})^l$\ \zerot\ for some $c\in \Z,$ $l\in \Z_{\geq 0}$. These states will be called infinity and zero states, respectively. We say that these are states of degree $c+2l$ mod $4$, since all exponents of $A$ in them have that value.

One can go from one state to another through crossing resolution changes, which may preserve the state type (zero or infinity) or change it. If it preserves the type,  then the change is between two states like in Fig. \ref{cross-change}(a) and the degrees of the two states involved coincide mod 4. (That is a consequence of a change $A \leftrightarrow A^{-1}$ related to the smoothing change and from a creation/elimination of the loop factor, $-A^2-A^{-2}$).  If a crossing resolution change affects the state type, then it is of the form (b) and it changes the state degrees by $2$ mod $4$. (Note that in this case the smoothing change cannot create or eliminate any loop.)
\begin{figure}
\centerline{\includegraphics[width=2.5in]{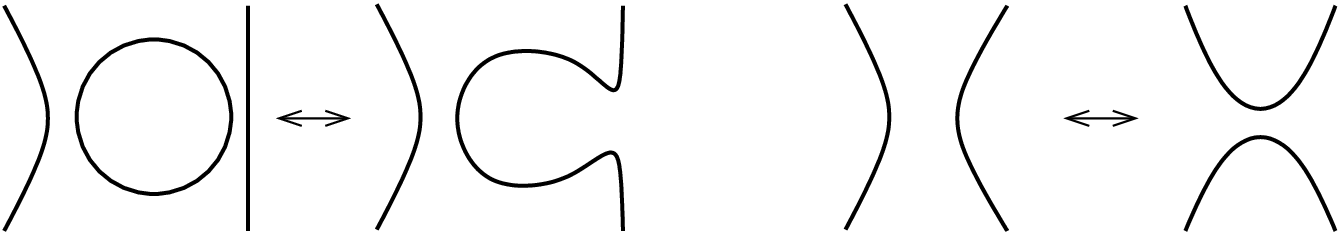}}
\caption{The effect of crossing smoothing change on states: (a) preserving type, (b) changing type}
\label{cross-change}
\end{figure}
Since one can go from one state to any other state through crossing resolution changes, we conclude that\\
(a) any two states of $\langle T\rangle$ of the same type are of the same degree mod $4$.\\
(b) any two states of $\langle T\rangle$ of different types are of degrees differing by $2$ mod $4$.\\
That proves the statement.
\end{proof}


We say that a tangle $T$, $\langle T\rangle =(a,b)$ is {\em algebraically trivializable} iff the ideal
$I(a,b)\triangleleft \Z[A^{\pm 1}]$ generated by $a,b$ equals $\Z[A^{\pm 1}].$ It is easy to see that this condition can be tested by
\begin{itemize}
\item computing the Gr\"obner basis of $I(A^n\cdot a,A^n\cdot b)$ in $\Z[A]$, where $n$ is the smallest exponent such that $A^n\cdot a,A^n\cdot b\in \Z[A]$, and then
\item  testing $A^n$ for a membership in that ideal.
\end{itemize}
In our program, we have used H. Kredel's JAS Java computer algebra system, which allows for computation of Gr\"obner bases over integers, \cite{Kr}.

The next statement shows that for testing the Jones conjecture it suffices to consider knot diagrams built of algebraically trivializable tangles only.

\begin{proposition}
If $T$ is a subtangle of a knot diagram with the trivial Jones polynomial, then $T$ is algebraically trivializable.
\end{proposition}

\begin{proof}
Suppose that $T'$ is the complement of $T$ in the knot diagram $D$, as in Fig. \ref{T-T'} with the trivial Jones polynomial. If
$T=(a,b),$ $T'=(a',b')$.
\begin{figure}
\centerline{\includegraphics[width=0.9in]{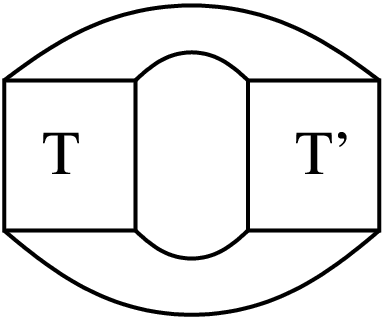}}
\caption{\ }
\label{T-T'}
\end{figure}
then the Kauffman bracket of the knot diagram is
$$\langle D\rangle=(aa'+bb')(-A^2-A^{-2})+ ab'+ba'=(-A)^{-3w(D)},$$
Therefore, $A^{3w(D)}\in I(a,b)$ and, hence, $T$ is algebraically trivializable.
\end{proof}

\section{Algebraic tangles}
Since we considered algebraic tangles as building blocks of knot diagrams with the trivial Jones polynomial, we
\begin{itemize}
\item discarded all algebraic tangles with internal loops.
\item considered algebraically trivializable tangles only.
\item considered tangles up to mutation only, because tangle mutations preserve knottedness of knots containing them, \cite{Ro}, and preserve Kauffman brackets.
\item generated algebraic tangles up to reflection and mirror image only, to
save time and storage.
\end{itemize}

Let $T_{p/q}$ be the rational tangle associated with $p/q.$
Then
$$r(T_{p/q})=T_{q/p}\quad \text{and}\quad T_{p/q}=\bar T_{-p/q}.$$
Therefore, in particular, we generated rational tangles $T_{p/q}$ for $p/q>1$ only.
It is known that they can be built of positive integral tangles, following the continuing fraction decomposition of $p/q.$

The numbers of algebraic tangle found are provided in the table below. The most intensive part of computations was checking algebraic trivializability of tangles.  For 17 crossing tangles that check took four weeks of computer time, and for 18 crossings, 32 weeks. Our results disagree with those of Yamada, \cite{Ya}. (We found more tangles).

\begin{table}[ht]
\caption{Number of algebraic tangles
(total and trivializable).}
{\begin{tabular}{|rrr|rrr|}
\hline
$n$ & Total & Triv & $n$ & Total & Triv \\
\hline
1  &         1 &       1 & 10 &      4334 &     2589 \\
2  &         1 &       1 & 11 &     15076 &     7754 \\
3  &         2 &       2 & 12 &     53648 &    23572 \\
4  &         4 &       4 & 13 &    193029 &    71124 \\
5  &        12 &      12 & 14 &    698590 &   211562 \\
6  &        36 &      30 & 15 &   2560119 &   633059 \\
7  &       113 &      94 & 16 &   9422500 &  1866458 \\
8  &       374 &     288 & 17 &  34935283 &  5478404 \\
9  &      1242 &     836 & 18 & 130250565 & 15674910 \\ \hline
\end{tabular}}
\label{algtang_table}
\end{table}

\section{Conway polyhedra}

A Conway polyhedron is a planar 4-valent graph, with no loops and no bigons
(and in particular, no bigon bounding the infinite region), see Fig. \ref{polyh} (center and left).
As observed by Conway \cite{Co}, every knot is either algebraic or
composed of algebraic tangles embedded in a Conway polyhedron, see examples in Fig. \ref{polyh}.  We say that a Conway polyhedron is {\em thin} if it can be disconnected by removing two of its edges. To optimize enumeration of knot diagrams necessary for verification of the Jones conjecture, non-thin Conway polyhedra and algebraic tangles were enumerated first and then reused repeatedly
during the main calculations.
This polyhedra enumeration was performed up to (abstract) isomorphism only, which is sufficient for our purposes by the following result:

\begin{proposition}
For the purpose of the verification of the Jones conjecture up to $n$ crossings it is sufficient to consider
(1) non-thin polyhedra only, and
(2) polyhedra up to an abstract isomorphism only.
\end{proposition}

\begin{proof}
(1) Thin polyhedra give rise to composite knots only. We claim that if
$K_{1} \# K_{2}$ given by a thin polyhedron
violates the Jones conjecture for non-trivial $K_{1}, K_{2}$
then so do $K_1$ and $K_2$. This follows from the fact that
$$J(K_{1} \# K_{2}) = J(K_1) \cdot J(K_2).$$
Hence, if $J(K_{1} \# K_{2})=1$
then $J(K_1)$ and $J(K_2)$ are monomials, and hence equal to 1 by
\cite[Cor. 3]{Gan}.
Finally, since the crossing numbers of diagrams $K_1$ and of $K_2$ are
less than that of $K_{1} \# K_{2}$,
the statement follows.

(2) By the above it is enough to consider 2-connected Conway polyhedra only. By \cite{ChG} any two abstractly isomorphic 2-connected polyhedra are related by a sequence of mutations. Since mutations preserve the Kauffman bracket and the unknottedness, it is enough to consider only one of these polyhedra for the purpose of verification of the Jones conjecture. Therefore, an enumeration of polyhedra up to abstract isomorphism is sufficient for our purposes.
\end{proof}

\begin{remark}
(1) The crossing number $cr(K)$ of a knot $K$
is the minimal crossing number among all diagrams of $K$.
It is conjectured that
$$cr(K_1 \# K_2)= cr(K_1)+cr(K_2),$$
cf. eg. \cite{Lac}.  Our argument above does not rely on that conjecture.

(2) Because the above Conway polyhedra are considered up to abstract
isomorphism only,
the numbers of Conway polyhedra differ from results shown in \cite{BrMcK}.
\end{remark}

\begin{figure}
\includegraphics[width=0.3\textwidth]{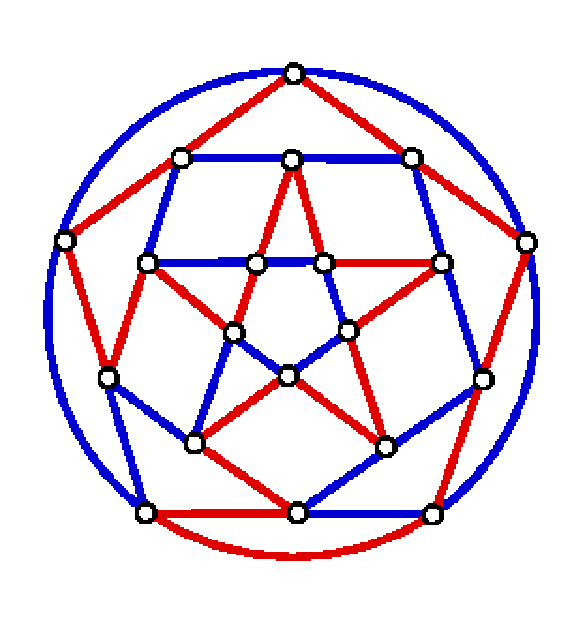} ~~~~~
\begin{tikzpicture}[scale=0.5]
\draw (0,3) -- (6,0) -- (5,6) -- (0,3) -- (3,2.5) -- (2,3.5) -- (4,3)
            -- (3,2.5) -- (6,0) -- (4,3) -- (5,6) -- (2,3.5) -- (0,3);
\node at (0,-1) {};
\end{tikzpicture}
~~~~~
\includegraphics[width=0.3\textwidth]{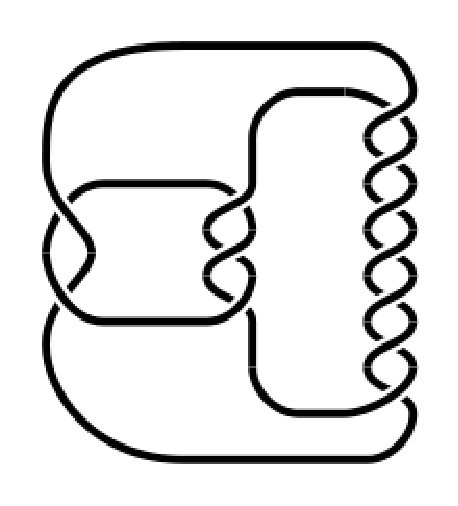}
\caption{Left and center: two examples of a Conway polyhedron.
Right: an algebraic knot called a pretzel knot.}
\label{polyh}
\end{figure}
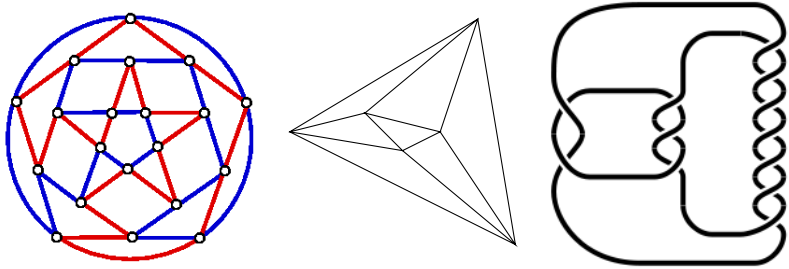

To facilitate double checking,
two independent approaches were used to enumerate Conway polyhedra:
\begin{enumerate}
\item Generation of simple 4-regular graphs using a backtracking
algorithm and software written in C by Meringer, \cite{Me}.  We modified the algorithm so that only planar graphs were generated, using the algorithm by Boyer and Myrvold
\cite{BMy} implemented in the Boost Graph Library, \cite{Bo}.
\item Generation of planar graphs using the program \verb+plantri+ written
by McKay, \cite{BMc}.
\end{enumerate}


\begin{figure}
\label{consum}
\includegraphics[width=0.45\textwidth]{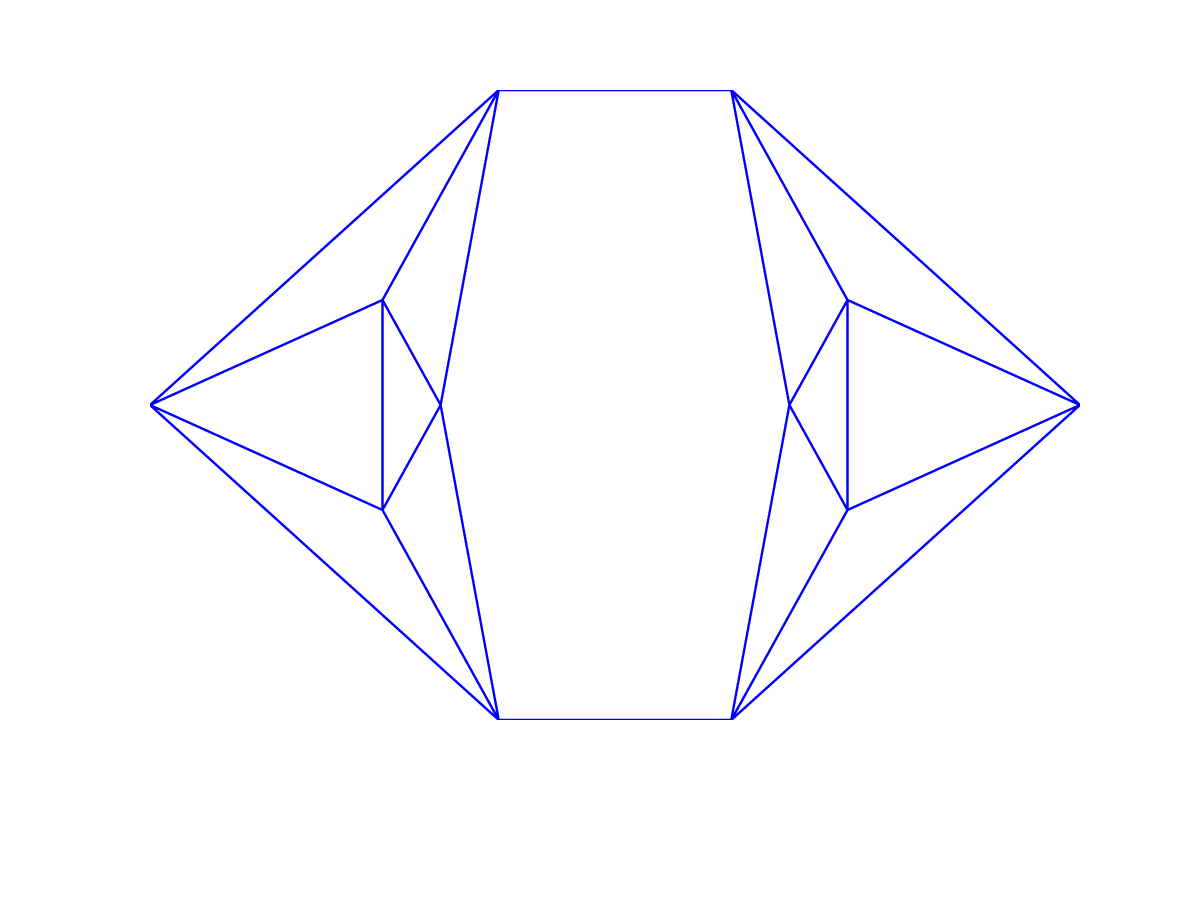} ~~~~~
\includegraphics[width=0.45\textwidth]{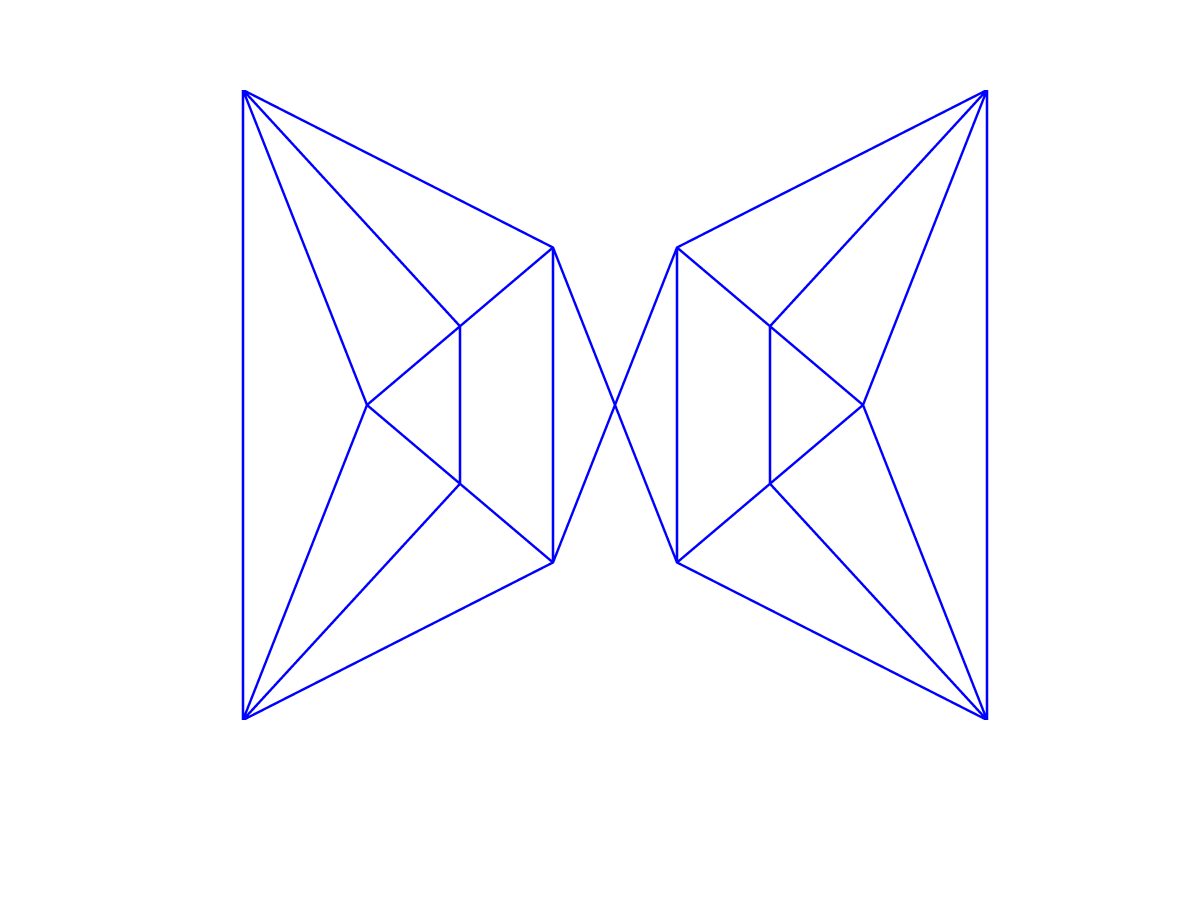}
\caption{Examples of thin Conway polyhedra
in which disconnecting edges emanate from
(a) different vertices and (b) the same vertex.}
\end{figure}

The polyhedron in the middle of Fig. \ref{polyh} is the only non-thin polyhedron with at most 6 vertices.
Numbers of Conway polyhedra are shown in Table \ref{polyh_table}.
It should be noted that the two approaches
agree with one another in number of graphs generated \cite{Me2} but
disagree with the results of Yamada.  Approximately 6-8\% of the Conway polyhedra with
14-22 vertices are thin.

\begin{table}[ht]
\caption{Number of Conway polyhedra (total, thin, and not thin).}
{\begin{tabular}{|r|rrr|}
\hline
$n$ & Total & Thin & Not Thin \\ \hline
 6 &      1 &     0 &      1 \\
 8 &      1 &     0 &      1 \\
 9 &      1 &     0 &      1 \\
10 &      3 &     0 &      3 \\
11 &      3 &     0 &      3 \\
12 &     13 &     1 &     12 \\
13 &     21 &     2 &     19 \\
14 &     68 &     5 &     63 \\
15 &    166 &    13 &    153 \\
16 &    543 &    44 &    499 \\
17 &   1605 &   132 &   1473 \\
18 &   5413 &   439 &   4974 \\
19 &  17735 &  1439 &  16296 \\
20 &  61084 &  4982 &  56102 \\
21 & 210221 & 17322 & 192899 \\
22 & 736287 & 61609 & 674678 \\
\hline
\end{tabular}}
\label{polyh_table}
\end{table}

\section{Computation of Kauffman brackets}
\label{s_KB}

The Kauffman bracket for a knot embedded in a Conway polyhedron with
$v$ vertices is related to the tangle Kauffman brackets
$(a_{i},b_{i})$ by
\begin{equation}\label{KB-state-sum}
\langle K \rangle = \sum_{\sigma} c_{1} c_{2} \cdots c_{v}
(-t^{1/2}-t^{-1/2})^{L(\sigma)-1}
\end{equation}
where $\sigma$ refers to one of the $2^{v}$ smoothings of the Conway
polyhedron, $c_{i}$ to $a_{i}$ or $b_{i}$, depending on the smoothing,
and $L(\sigma)$ to the number of loops in the smoothing.

The above summation formula is very computationally demanding, see next section.
Fortunately, the state sum approach can be improved upon significantly by a step by step divide-and-conquer method, in which one proceeds by computing Kauffman bracket for subtangles of a knot first. (A similar idea can be found in \cite{BN}.) For that purpose we utilize the concept of the Kauffman bracket of an  $n$-tangle which is a straightforward generalization of the Kauffman bracket of a 2-tangle.
The main difference being that it takes values in
the relative skein module of a disk with $n$ boundary points, see \cite{Pr1}, which is
 the free module $\Z[A^{\pm 1}]M_n,$ with the basis $M_n$ given by the crossingless matchings of $n$ points on a circle.
Note that $|M_n|$ is the Catalan number
$C_{m}=\frac{1}{m+1}{2m \choose m}$ for $m=n/2.$



The step by step method can add vertices in any order in a Conway
polyhedron $P$, e.g. 1-4-6-5-2-3 in the 6-vertex polyhedron.
We then consider a sequence of subtangles $T_{n}$ of $P$ obtained by taking the
subpolyhedron of $P$ composed of the first $n$ vertices in this vertex
sequence. For computational efficiency, we choose the above sequence so that each sub-polyhedron is connected and the number of dangling
edges of each of the subtangles is minimal.

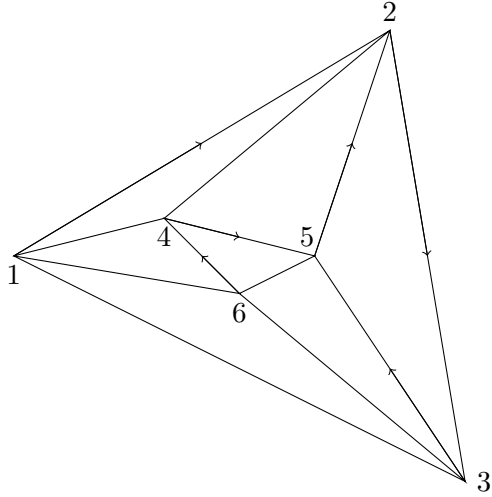
\begin{figure}\label{slices}
\label{pieces}
~~~~~~~~~~~~~~~~~~~~
\begin{tikzpicture}
\draw (0,3) -- (6,0) -- (5,6) -- (0,3) -- (3,2.5) -- (2,3.5) -- (4,3)
            -- (3,2.5) -- (6,0) -- (4,3) -- (5,6) -- (2,3.5) -- (0,3);
\draw [->] (0,3) -- (2.5,4.5);
\draw [->] (2,3.5) -- (3,3.25);
\draw [->] (3,2.5) -- (2.5,3);
\draw [->] (4,3) -- (4.5,4.5);
\draw [->] (5,6) -- (5.5,3);
\draw [->] (6,0) -- (5,1.5);
\node at (0,2.75) {1};
\node at (5,6.25) {2};
\node at (6.25,0) {3};
\node at (2,3.25) {4};
\node at (3.9,3.25) {5};
\node at (3,2.25) {6};
\end{tikzpicture}
\caption{The 6-vertex polyhedron. NW corners for each vertex shown by outgoing arrows.  (For example, the NW corner for vertex 1 attaches to vertex 2.)}
\end{figure}


Now we compute Kauffman brackets of successive tangles $T_n.$
In the example above,
\begin{equation}
\begin{tikzpicture}[scale=0.03528]
\draw (0,0) -- (15,-8);
\draw (0,0) -- (15,0);
\draw (0,0) -- (15,8);
\draw [->] (0,0) -- (15,15);
\end{tikzpicture}
\begin{tikzpicture}[scale=0.03528]
\draw(-38,9) node {$= t^{-1/2}Q_1$};
\draw (0,18) arc [radius=9, start angle=90, end angle=270];
\draw (0,12) arc [radius=3, start angle=90, end angle=270];
\end{tikzpicture}
\begin{tikzpicture}[scale=0.03528]
\draw(-20,9) node {$+Q_2$};
\draw (0,18) arc [radius=3, start angle=90, end angle=270];
\draw (0,6) arc [radius=3, start angle=90, end angle=270];
\end{tikzpicture}
\end{equation}
where
$Q_1=p_1,$ $Q_2=q_1,$ for some $p_1,q_1\in \Z[t^{\pm 1}],$
by Lemma \ref{KB-tangle}. The arrow denotes the NW corner of vertex 1.  Next we compute
$\langle T_2\rangle$ from $\langle T_1\rangle:$
\begin{equation}
\begin{tikzpicture}[scale=0.03528]
\draw (0,18) arc [radius=9, start angle=90, end angle=270];
\draw (0,12) arc [radius=3, start angle=90, end angle=270];
\draw (0,12) -- (6,12) -- (18,18);
\draw [->] (6,12) -- (18,12);
\draw (6,12) -- (18,6);
\end{tikzpicture}
\begin{tikzpicture}[scale=0.03528]
\draw(-40,15) node {$=t^{-1/2} p_{4}$\quad };
\draw (0,30) arc [radius=15, start angle=90, end angle=270];
\draw (0,24) arc [radius=3, start angle=90, end angle=270];
\draw (0,12) arc [radius=3, start angle=90, end angle=270];
\end{tikzpicture}
\begin{tikzpicture}[scale=0.03528]
\draw(-25,15) node {$+q_{4}$};
\draw (0,30) arc [radius=15, start angle=90, end angle=270];
\draw (0,24) arc [radius=9, start angle=90, end angle=270];
\draw (0,18) arc [radius=3, start angle=90, end angle=270];
\end{tikzpicture}
\end{equation}
\begin{equation}
\begin{tikzpicture}[scale=0.03528]
\draw (0,18) arc [radius=3, start angle=90, end angle=270];
\draw (0,6) arc [radius=3, start angle=90, end angle=270];
\draw (0,12) -- (6,12) -- (18,18);
\draw [->] (6,12) -- (18,12);
\draw (6,12) -- (18,6);
\end{tikzpicture}
\begin{tikzpicture}[scale=0.03528]
\draw(-40,15) node {$=t^{-1/2} p_{4}$};
\draw (0,30) arc [radius=9, start angle=90, end angle=270];
\draw (0,24) arc [radius=3, start angle=90, end angle=270];
\draw (0,6) arc [radius=3, start angle=90, end angle=270];
\end{tikzpicture}
\begin{tikzpicture}[scale=0.03528]
\draw(-25,15) node {$+q_{4}$};
\draw (0,30) arc [radius=3, start angle=90, end angle=270];
\draw (0,18) arc [radius=3, start angle=90, end angle=270];
\draw (0,6) arc [radius=3, start angle=90, end angle=270];
\end{tikzpicture}
\end{equation}
where the arrow denotes the NW corner of vertex 4.  Thus,
\begin{equation}
\hspace*{-.3in}
\begin{tikzpicture}[scale=0.03528]
\draw (0,0) -- (45,-24);
\draw (0,0) -- (45,-16);
\draw (30,0) -- (45,-8);
\draw [->] (0,0) -- (45,0);
\draw (30,0) -- (45,8);
\draw [->] (0,0) -- (45,16);
\end{tikzpicture}
\begin{tikzpicture}[scale=0.03528]
\draw(-30,15) node {$=R_1$};
\draw(-30,-5) node {};
\draw (0,30) arc [radius=15, start angle=90, end angle=270];
\draw (0,24) arc [radius=3, start angle=90, end angle=270];
\draw (0,12) arc [radius=3, start angle=90, end angle=270];
\end{tikzpicture}
\begin{tikzpicture}[scale=0.03528]
\draw(-40,15) node {$+t^{-1/2}R_2$};
\draw(-30,-5) node {};
\draw (0,30) arc [radius=15, start angle=90, end angle=270];
\draw (0,24) arc [radius=9, start angle=90, end angle=270];
\draw (0,18) arc [radius=3, start angle=90, end angle=270];
\end{tikzpicture}
\begin{tikzpicture}[scale=0.03528]
\draw(-35,15) node {$+t^{-1/2}R_3$};
\draw(0,-5) node {};
\draw (0,30) arc [radius=9, start angle=90, end angle=270];
\draw (0,24) arc [radius=3, start angle=90, end angle=270];
\draw (0,6) arc [radius=3, start angle=90, end angle=270];
\end{tikzpicture}
\begin{tikzpicture}[scale=0.03528]
\draw(-25,15) node {$+R_4$};
\draw(-30,-5) node {};
\draw (0,30) arc [radius=3, start angle=90, end angle=270];
\draw (0,18) arc [radius=3, start angle=90, end angle=270];
\draw (0,6) arc [radius=3, start angle=90, end angle=270];
\end{tikzpicture}
\end{equation}
where $R_1=Q_1 t^{-1} p_4,$ $R_2=Q_2p_4,$ $R_3=Q_1q_4,$ $R_4=Q_2q_4.$\\
$\langle T_3\rangle$ is computed in a similar manner:
\begin{equation}
\begin{tikzpicture}[scale=0.03528]
\draw (0,30) arc [radius=15, start angle=90, end angle=270];
\draw (0,24) arc [radius=3, start angle=90, end angle=270];
\draw (0,12) arc [radius=3, start angle=90, end angle=270];
\draw (0,6) -- (12,6);
\draw (0,12) -- (12,12);
\draw [->] (24,6) -- (12,12);
\draw (24,12) -- (12,6);
\end{tikzpicture}
\begin{tikzpicture}[scale=0.03528]
\draw(-45,15) node {$=t^{-1/2} p_{6} \epsilon$};
\draw (0,30) arc [radius=15, start angle=90, end angle=270];
\draw (0,24) arc [radius=3, start angle=90, end angle=270];
\draw (0,12) arc [radius=3, start angle=90, end angle=270];
\end{tikzpicture}
\begin{tikzpicture}[scale=0.03528]
\draw(-30,15) node {$+ q_{6}$};
\draw (0,30) arc [radius=15, start angle=90, end angle=270];
\draw (0,24) arc [radius=3, start angle=90, end angle=270];
\draw (0,12) arc [radius=3, start angle=90, end angle=270];
\end{tikzpicture}
\end{equation}
\begin{equation}
\begin{tikzpicture}[scale=0.03528]
\draw (0,30) arc [radius=15, start angle=90, end angle=270];
\draw (0,24) arc [radius=9, start angle=90, end angle=270];
\draw (0,18) arc [radius=3, start angle=90, end angle=270];
\draw (0,6) -- (12,6);
\draw (0,12) -- (12,12);
\draw [->] (24,6) -- (12,12);
\draw (24,12) -- (12,6);
\end{tikzpicture}
\begin{tikzpicture}[scale=0.03528]
\draw(-40,15) node {$=t^{-1/2} p_{6}$};
\draw (0,30) arc [radius=15, start angle=90, end angle=270];
\draw (0,24) arc [radius=3, start angle=90, end angle=270];
\draw (0,12) arc [radius=3, start angle=90, end angle=270];
\end{tikzpicture}
\begin{tikzpicture}[scale=0.03528]
\draw(-30,15) node {$+ q_{6}$};
\draw (0,30) arc [radius=15, start angle=90, end angle=270];
\draw (0,24) arc [radius=9, start angle=90, end angle=270];
\draw (0,18) arc [radius=3, start angle=90, end angle=270];
\end{tikzpicture}
\end{equation}
\begin{equation}
\begin{tikzpicture}[scale=0.03528]
\draw (0,30) arc [radius=9, start angle=90, end angle=270];
\draw (0,24) arc [radius=3, start angle=90, end angle=270];
\draw (0,6) arc [radius=3, start angle=90, end angle=270];
\draw (0,6) -- (12,6);
\draw (0,12) -- (12,12);
\draw [->] (24,6) -- (12,12);
\draw (24,12) -- (12,6);
\end{tikzpicture}
\begin{tikzpicture}[scale=0.03528]
\draw(-40,15) node {$=t^{-1/2} p_{6}$};
\draw (0,30) arc [radius=15, start angle=90, end angle=270];
\draw (0,24) arc [radius=3, start angle=90, end angle=270];
\draw (0,12) arc [radius=3, start angle=90, end angle=270];
\end{tikzpicture}
\begin{tikzpicture}[scale=0.03528]
\draw(-25,15) node {$+ q_{6}$};
\draw (0,30) arc [radius=9, start angle=90, end angle=270];
\draw (0,24) arc [radius=3, start angle=90, end angle=270];
\draw (0,6) arc [radius=3, start angle=90, end angle=270];
\end{tikzpicture}
\end{equation}
\begin{equation}
\begin{tikzpicture}[scale=0.03528]
\draw (0,30) arc [radius=3, start angle=90, end angle=270];
\draw (0,18) arc [radius=3, start angle=90, end angle=270];
\draw (0,6) arc [radius=3, start angle=90, end angle=270];
\draw (0,6) -- (12,6);
\draw (0,12) -- (12,12);
\draw [->] (24,6) -- (12,12);
\draw (24,12) -- (12,6);
\end{tikzpicture}
\begin{tikzpicture}[scale=0.03528]
\draw(-35,15) node {$=t^{-1/2} p_{6}$};
\draw (0,30) arc [radius=3, start angle=90, end angle=270];
\draw (0,18) arc [radius=9, start angle=90, end angle=270];
\draw (0,12) arc [radius=3, start angle=90, end angle=270];
\end{tikzpicture}
\begin{tikzpicture}[scale=0.03528]
\draw(-20,15) node {$+ q_{6}$};
\draw (0,30) arc [radius=3, start angle=90, end angle=270];
\draw (0,18) arc [radius=3, start angle=90, end angle=270];
\draw (0,6) arc [radius=3, start angle=90, end angle=270];
\end{tikzpicture}
\end{equation}
where the factor of $\epsilon = -t^{1/2} - t^{-1/2}$ comes from a loop
closure. Consequently,
\begin{equation}
\begin{tikzpicture}[scale=0.03528]
\draw (0,0) -- (69,-30);
\draw (0,0) -- (46,-12);
\draw (46,-12) -- (69,-18);
\draw (46,-12) -- (69,-6);
\draw (46,-12) -- (23,6);
\draw [->] (46,-12) -- (34.5,-3);
\draw (0,0) -- (23,6);
\draw [->] (23,6) -- (69,6);
\draw (23,6) -- (69,18);
\draw [->] (0,0) -- (69,30);
\end{tikzpicture}
\begin{tikzpicture}[scale=0.03528]
\draw(-30,15) node {$=S_{1}$};
\draw(-30,-12) node {};
\draw (0,30) arc [radius=15, start angle=90, end angle=270];
\draw (0,24) arc [radius=3, start angle=90, end angle=270];
\draw (0,12) arc [radius=3, start angle=90, end angle=270];
\end{tikzpicture}
\begin{tikzpicture}[scale=0.03528]
\draw(-40,15) node {$+t^{-1/2}S_{2}$};
\draw(-30,-12) node {};
\draw (0,30) arc [radius=15, start angle=90, end angle=270];
\draw (0,24) arc [radius=9, start angle=90, end angle=270];
\draw (0,18) arc [radius=3, start angle=90, end angle=270];
\end{tikzpicture}
\begin{tikzpicture}[scale=0.03528]
\draw(-35,15) node {$+t^{-1/2}S_{3}$};
\draw(-30,-12) node {};
\draw (0,30) arc [radius=9, start angle=90, end angle=270];
\draw (0,24) arc [radius=3, start angle=90, end angle=270];
\draw (0,6) arc [radius=3, start angle=90, end angle=270];
\end{tikzpicture}
\begin{tikzpicture}[scale=0.03528]
\draw(-35,15) node {$+t^{-1/2}S_{4}$};
\draw(-30,-12) node {};
\draw (0,30) arc [radius=3, start angle=90, end angle=270];
\draw (0,18) arc [radius=9, start angle=90, end angle=270];
\draw (0,12) arc [radius=3, start angle=90, end angle=270];
\end{tikzpicture}
\begin{tikzpicture}[scale=0.03528]
\draw(-20,15) node {$+S_{5}$};
\draw(-30,-12) node {};
\draw (0,30) arc [radius=3, start angle=90, end angle=270];
\draw (0,18) arc [radius=3, start angle=90, end angle=270];
\draw (0,6) arc [radius=3, start angle=90, end angle=270];
\end{tikzpicture}
\end{equation}
where
\begin{equation}
S_{1} = R_1 (-1 - t) p_6+R_1q_6+R_2 t^{-1} p_6+R_3 t^{-1} p_6,
\end{equation}
\begin{equation}
S_{2} = R_2q_6,
\quad S_{3} = R_3q_6,
\quad S_{4} = R_4p_6,
\quad S_{5} = R_4q_6
\end{equation}
This procedure continues for another three steps (vertices).

\section{Computational complexity of Kauffman bracket computation}

The computational complexity of the state sum (\ref{KB-state-sum}) and of our step by step methods stems from a large number of multiplications of Laurent polynomials $p_1,...,p_v,q_1,...,q_v$, where $v$ is the number of vertices in the Conway polyhedron. (Note for example that multiplying two Laurent polynomials of span 3 requires 16 monomial multiplications and 9 additions.)
Therefore, to compare these two methods let us consider the numbers of such multiplications involved $N_1(v)$ in the first method and $N_2(v)$ in the second, both as functions of $v$. (For simplicity, we are not counting additions, multiplications by powers of $t$ and by the loop factors $\epsilon=-t^2-t^{-2}.$)

Since (\ref{KB-state-sum}) has $2^{v}$ terms, $N_1(v)=2^v(v-1)$.

In our step by step method, there are $2C_{n_{k-1}/2}$ multiplications at $k$-th step, for $k=1,...,v,$ where $n_{k-1}$ is the number of the endpoints of the $k$-$1$-th partial tangle $T_{k-1}.$ Therefore,  $N_2(v)=\sum_{k=1}^v 2C_{n_{k-1}/2}.$  Note that $n_k\leq n_{k-1}+2$ for every $k$ and, hence, $n_k\leq 2k$ and, similarly, $n_k\leq 2(v-k)$. Therefore
$$N_2(v)\leq 4\sum_{i=1}^{\lfloor v/2\rfloor}C(i).$$
That is much smaller than $N_1(v)$ for small $v$. For example, $N_2(4)\leq 12$ versus $N_1(4)=2^4(4-1)=48$.
For large $v$, we have
$$N_2(v)\leq 4\sum_{i=1}^{\lfloor v/2\rfloor}C(i) \leq \frac{80\cdot 4^{\lfloor v/2\rfloor}}{9\sqrt{e\pi} \lfloor v/2\rfloor^{3/2}}
,$$
\cite{Catalan},
which grows slower in $v$ than $2^v (v-1)$.

The above formulas for $N_1(v)$ and $N_2(v)$ do not take into account the
dependence of the computational effort of the Kauffman bracket computation on the complexity of multiplications of $p_1,...,p_v, q_1,...,q_v,$ which increases with the numbers of monomials in these Laurent polynomials.
(The product of two polynomials involving $n_1$ and $n_2$ monomials respectively requires $n_1\cdot n_2$ monomial multiplications.)

In a final note, let us mention that there is no polynomial time algorithm in the number of crossings for computing the Kauffman bracket. \cite{JVW} shows that every algorithm for computing Jones polynomial is \#P-hard.


\section{Pretesting diagrams with $t = -1$}

As we have observed above most of the computational effort in computing the Jones polynomial goes into polynomial multiplications.  However,
if the Laurent polynomials are first evaluated at a certain value of $t$ and stored
as floating point numbers for later use, polynomial multiplications are
replaced by much less intensive single floating point operations.
For that reason we have pretested all knot diagrams for triviality of their Jones
polynomial by computing their value at $t=-1$. That value, yielding the knot determinant, appears to be optimal for several reasons:
\begin{itemize}
\item Unlike for $t=1,i,e^{2\pi i/3},$ the Jones polynomial at $t=-1$ has an unlimited  range of values.
\item The $t=-1$ test inexpensively eliminates a large fraction of candidates for non-triviality of the Jones polynomial, relative to full computation of the Kauffman bracket. Among $802$ knots up to 11 crossings, the only ones with determinant $1$ are  $10_{124}, 10_{153}, 11n_{34}, 11n_{42}, 11n_{49}$ and $11n_{116},$ see \cite{KA}.
\item The computations of the Jones polynomial for $t=-1$ can be precise for up to at least 53 crossings in IEEE double precision, a floating point system with 53 bits of precision.
\item At $t = -1$, $\epsilon = 0$, making for shorter computation sequences.
\end{itemize}

\section{Elimination of diagrams with crossing-reducing or simplifying
pass moves}
\label{s_pass}

A {\em pass move} in a link diagram moves a strand with successive over-crossings or under-crossings to another location in the diagram, see Fig. \ref{pass6}.  In the process, the number of crossings
may change.


If a knot diagram has a crossing decreasing pass move,
then that diagram need not be considered for testing of the Jones conjecture, because another diagram of that knot with fewer crossings was tested already.
Also, if a diagram has a pass move that reduces the number of vertices in
the Conway polyhedron, then that diagram need not be considered, since an
equivalent diagram in a smaller Conway polyhedron was already considered.

\begin{figure}
\includegraphics[width=0.45\textwidth]{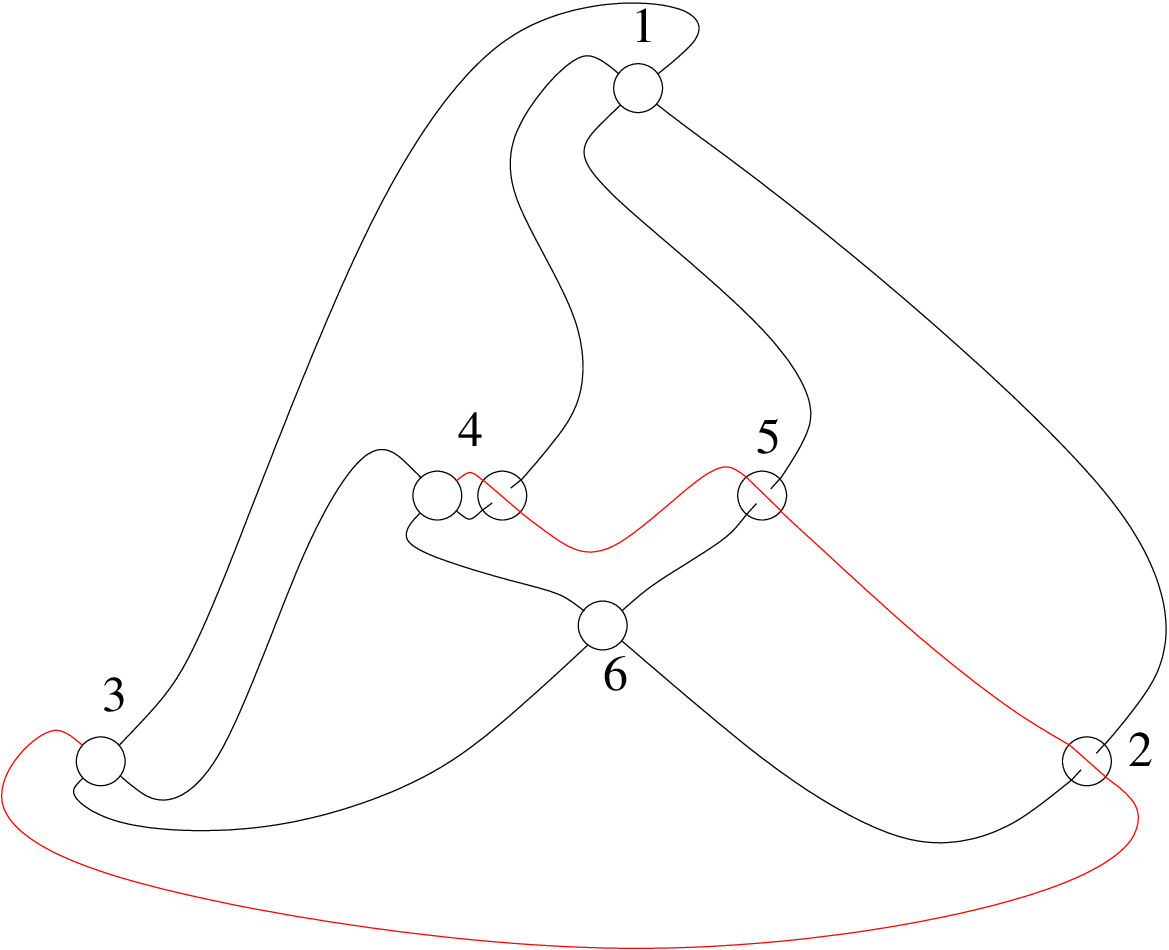}
\includegraphics[width=0.45\textwidth]{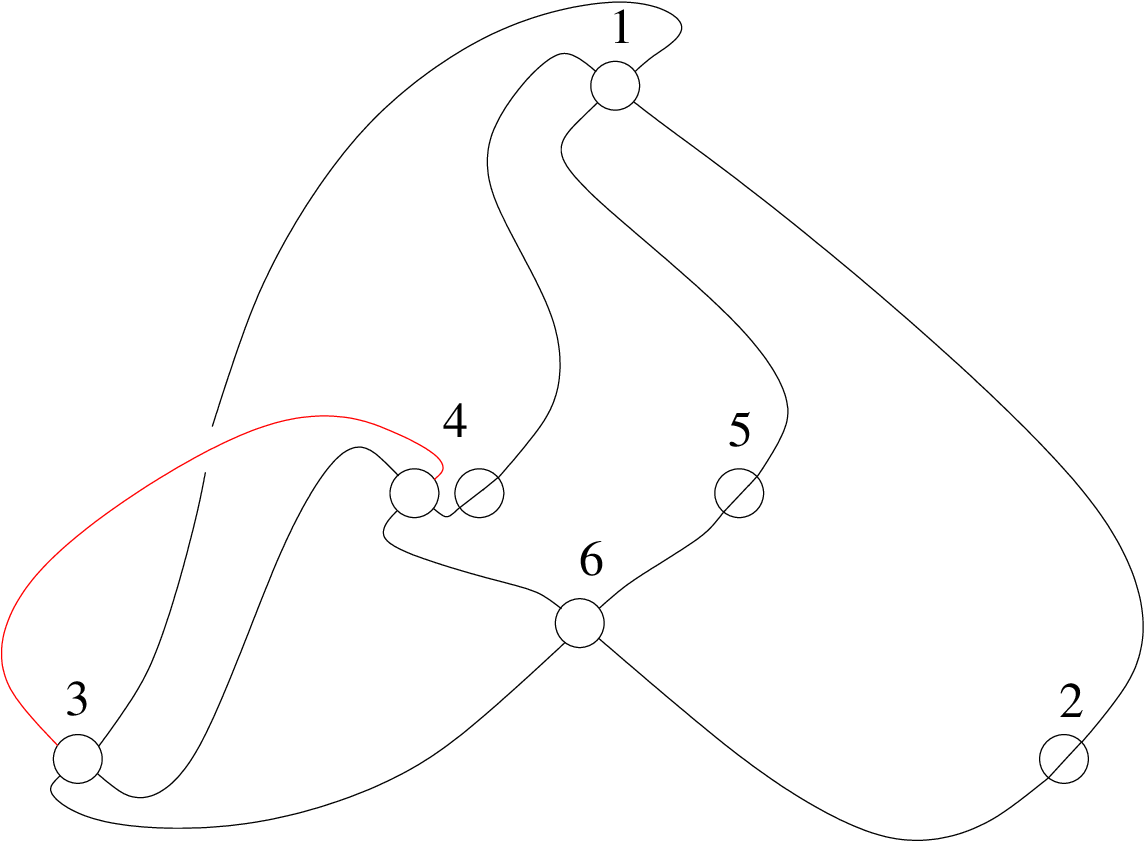}
\caption{Example of a (3,1)-pass move in the 6-vertex polyhedron
Left: before pass move; right: after pass move.}
\label{pass6}
\end{figure}

Redirecting a strand over vertices 4, 5, and 2 in Fig. \ref{pass6}  (left),
over vertices 3, 4 (right)
results in a diagram with fewer crossings, regardless of contents of vertices 3, 4, and 6.

It is known that there are unknot diagrams without any crossing reducing pass move, see eg. \cite{Ya}.  In our computer enumeration of knot diagrams, we found some new examples of such unknot diagrams as well, for example Fig. \ref{suspect}.

Interestingly, the above pass move simplification method eliminated all Jones conjecture counterexample candidates based on the 6-vertex polyhedron, up to 22 crossings.





\section{Testing trivial Jones polynomial diagrams for unknottedness}
\label{s_reco}

Testing trivial Jones polynomial diagrams for unknottedness was achieved by SnapPy, by computing knot group presentations. We relied here on the extremely high efficiency of SnapPy in providing the minimal (single generator, no relations)  presentation for the unknot diagrams. Nonetheless, some diagrams presented a challenge for SnapPy. Fig. \ref{suspect} shows a  trivial Jones polynomial diagram which SnapPy was unable to recognize as an unknot, using a coordinates and crossing representation of the knot as input. It returned the following presentation of its knot group:
$$\langle a,b,c \mid   acb^{-1}c^{-1}b^{-1}ca^2cb^{-1}c^{-1}b^{-1}cabcbc^{-1}b^{-1}a^{-1}c^{-1}bcbc^{-1}a^{-2}b,$$
$$\quad   aba^{-1}bacb^{-1}c^{-1}b^{-1}caba^{-1}ba^{-1}bacb^{-1}c^{-1}b^{-1}c\rangle.$$
(However, SnapPy was able to recognize this diagram as the trivial knot by its Dowker code.)
As we have mentioned earlier, this diagram doesn't admit a crossing reducing pass move.

\begin{figure}
\center{\includegraphics[width=2.5in]{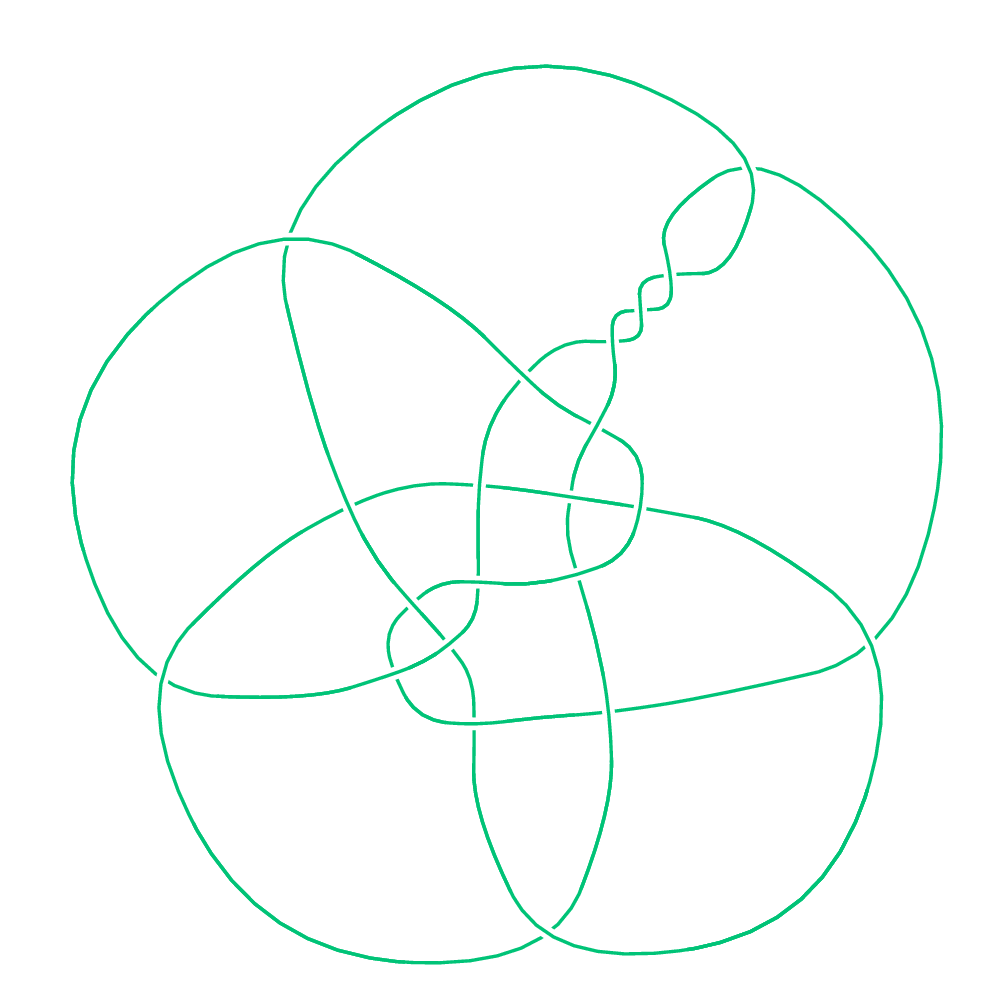}}
\caption{An unknot diagram with no crossing reducing pass moves}
\label{suspect}
\end{figure}

\section{Computational aspects of the project}
\label{s_computation}

In the process of verifying the Jones conjecture up to 22 crossings, we tested  2,257,956,747,340 non-algebraic knot diagrams. It took 31.9 days of CPU time on an Intel i7-4790 4-core desktop machine to generate all algebraically trivializable non-algebraic knot diagrams up to 21 crossings and to compute their determinants. After that, it took 1.7 days to compute Kauffman brackets of those with determinant $1$.  (The elapsed time was much shorter because computations were run on several cores at the same time.)  Computations for 22 crossings were performed on 8-core Intel Xeon L5520 processors operated by the Center for Computational Research at the University at Buffalo and took 439.2 core-days of wall clock time.

A breakdown of the CPU times by run is shown in Fig. \ref{times}. (Please note that 22 crossing results
were obtained on Intel Xeon processor, which was somewhat slower than our desktop Intel i7 one.)
Not surprisingly, the times rise
approximately exponentially with crossing numbers.  With respect to the
number of vertices, two competing trends contribute to the overall trend:
the increasing number of Conway polyhedra with number of vertices, and
the decreasing number of knot diagrams with increasing vertex number
(within a given crossing number).  
Numbers of diagrams considered (Fig. \ref{num_test_cand}) are also consistent with these trends.

The algebraic knot calculations took, in total, less than a day.
%
The testing for unknottedness consisting of using SnapPy to check the
knot group presentation took approximately 3.7 hours.

Preceding the above Jones conjecture testing routines, significant CPU times were required for Conway polyhedron generation (about 60 days) and a generation of algebraic tangles. By far the most extensive portion of the algebraic tangle
calculations were the trivializability calculations, which took about four weeks
of CPU time up to 17 crossings, and 32 weeks of CPU time for 18 crossings.

\begin{figure}]h]
\includegraphics[width=0.95\textwidth]{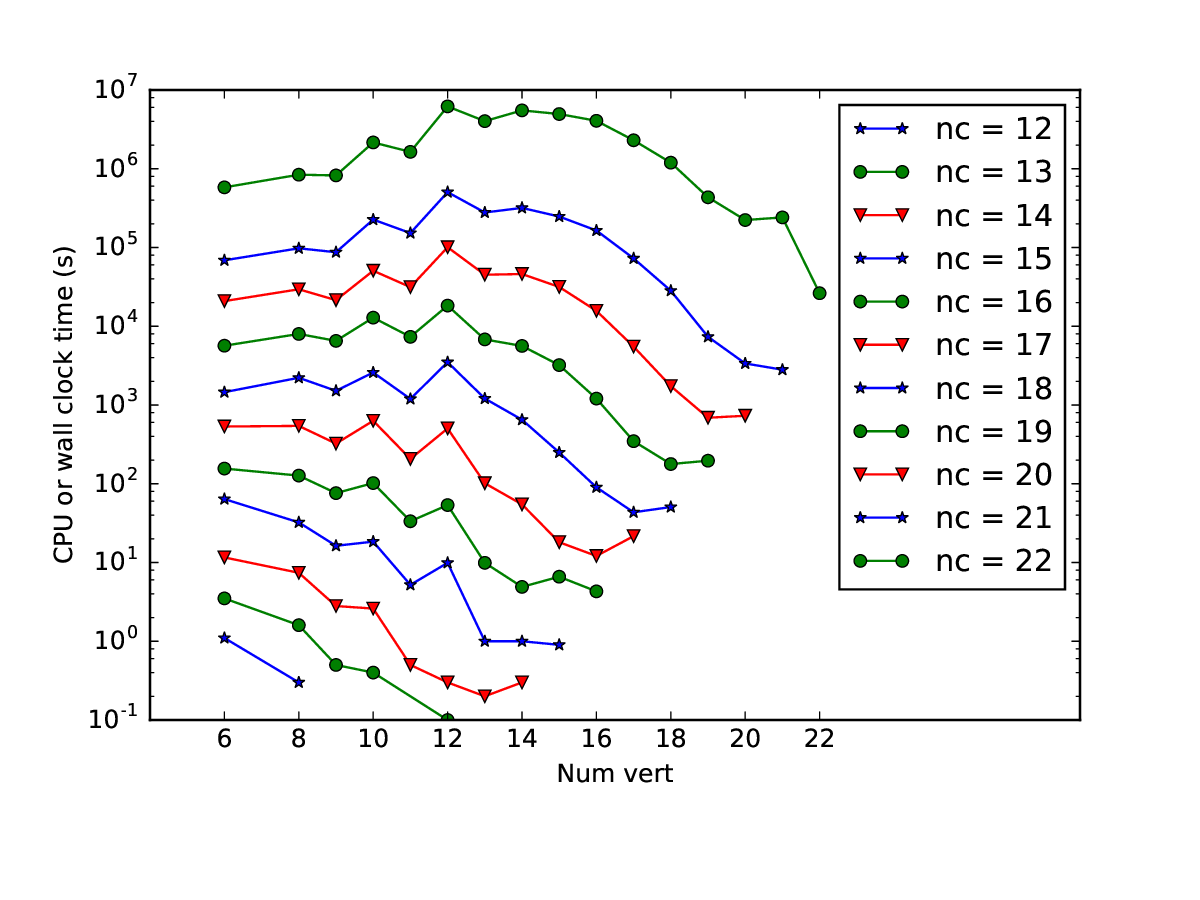} \\
\includegraphics[width=0.95\textwidth]{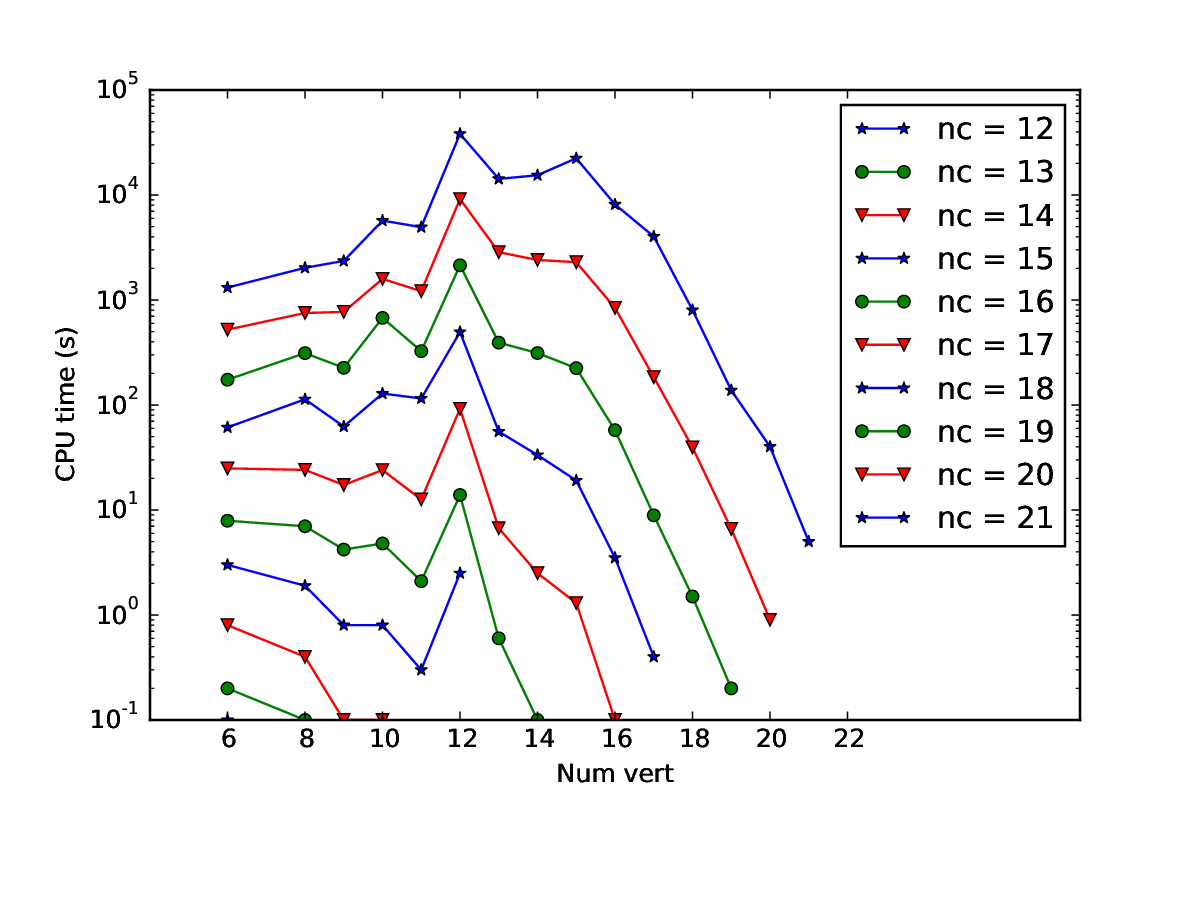}
\caption{Top: CPU times for non-algebraic knot diagrams generation and determinant testing.
Bottom: CPU times for Kauffman bracket computations of non-algebraic determinant 1 diagrams.}
\label{times}
\end{figure}

\begin{figure}[h]
\includegraphics[width=0.95\textwidth]{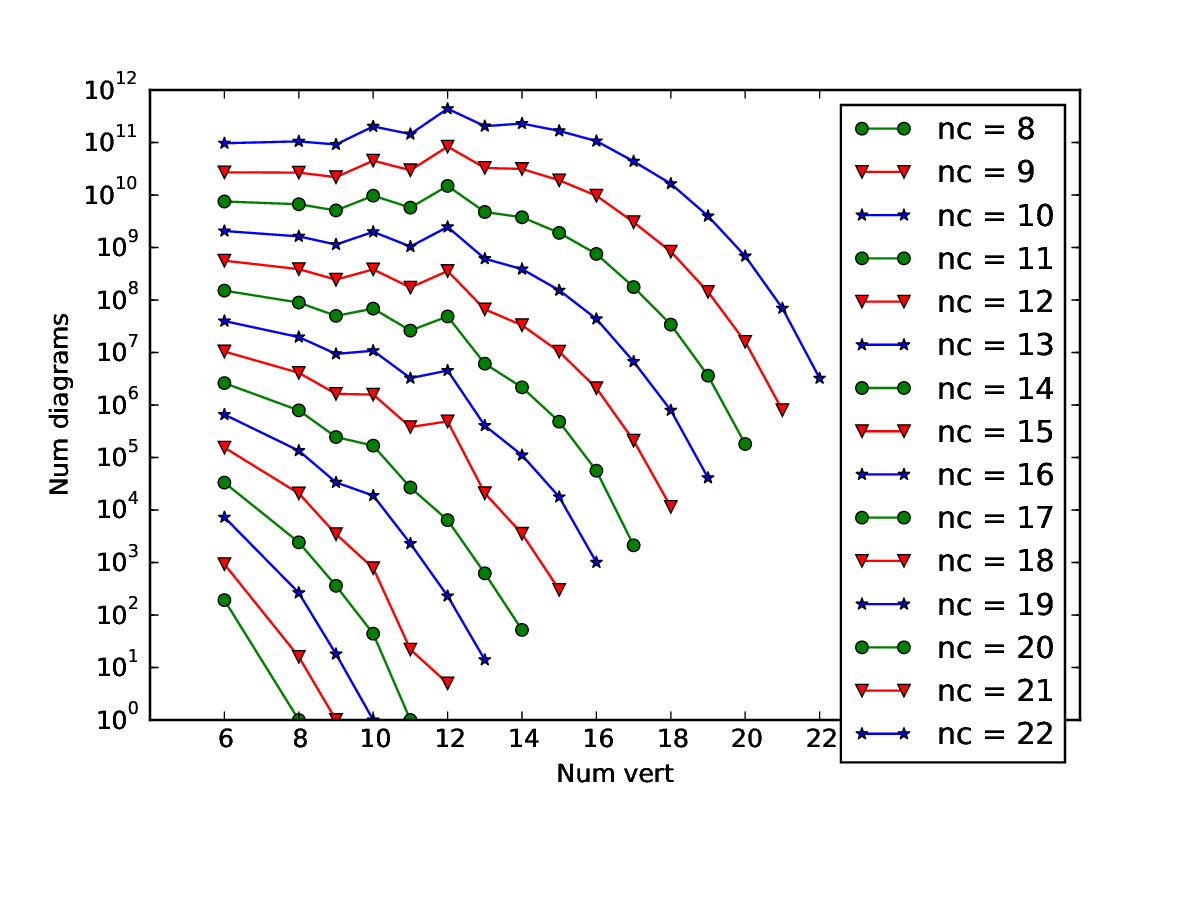} \\
\includegraphics[width=0.95\textwidth]{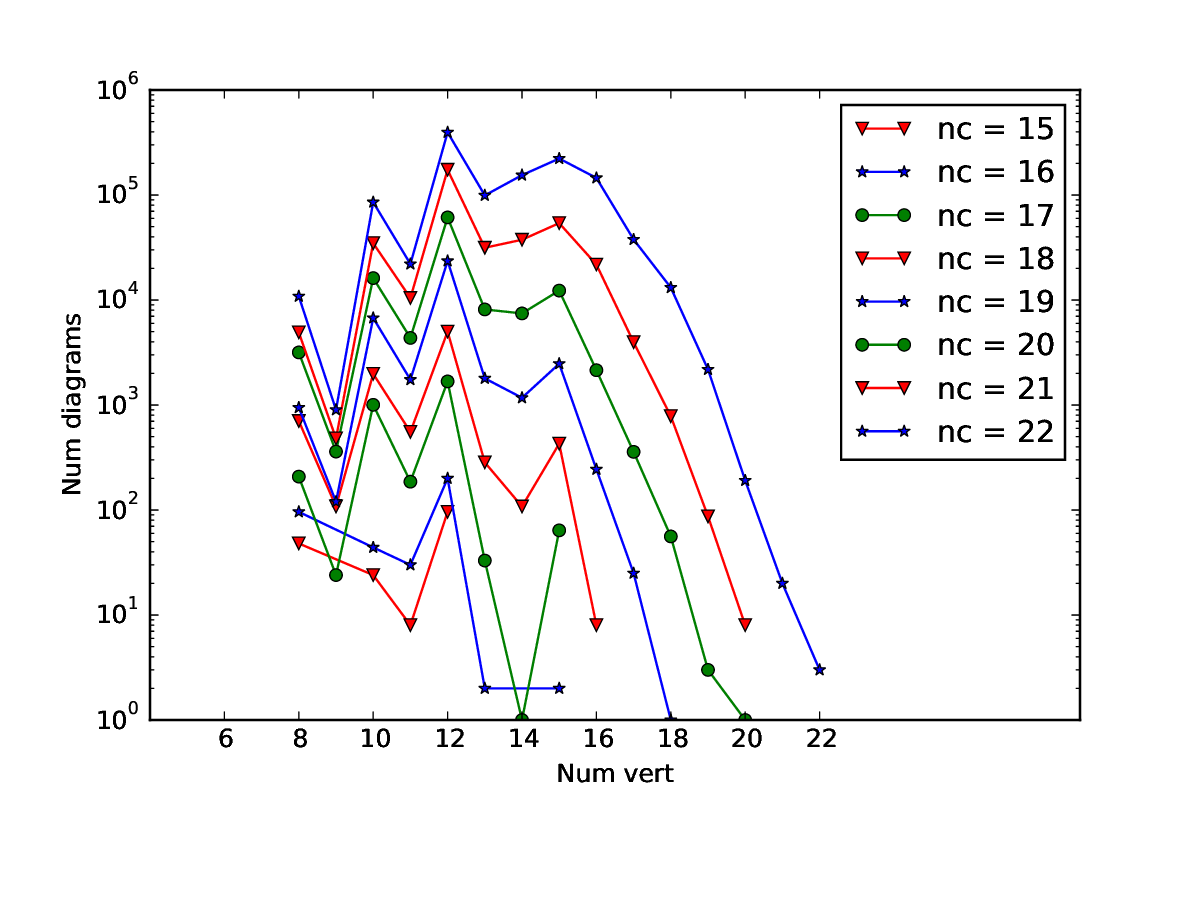}
\caption{Top: Number of non-algebraic knot diagrams tested for determinant.
Bottom: Number of diagrams with monomial Kauffman bracket tested for unknottedness.}
\label{num_test_cand}
\end{figure}

\end{document}